\documentclass[11pt, oneside]{amsart} 

\usepackage[british,english]{babel} 
\usepackage{graphics,color,pgf}
\usepackage{epsfig}
\usepackage[ansinew]{inputenc}
\usepackage[all]{xy}
\usepackage{hyperref}
\newdir{ >}{!/8pt/@{}*@{>}}
\usepackage{amssymb, amsmath,amsthm, mathtools, amscd}
\usepackage{mathrsfs}
\usepackage{stmaryrd}
\usepackage[margin=1.5in]{geometry}

\makeatletter
\newtheorem*{rep@theorem}{\rep@title}
\newcommand{\newreptheorem}[2]{%
\newenvironment{rep#1}[1]{%
 \def\rep@title{#2 \ref{##1}}%
 \begin{rep@theorem}}%
 {\end{rep@theorem}}}
\makeatother

\theoremstyle{plain}
\newtheorem{teor}{Theorem}[section]
\newreptheorem{teor}{Theorem}  
\newtheorem{lem}[teor]{Lemma}
\newtheorem{cor}[teor]{Corollary}
\newtheorem{prop}[teor]{Proposition}
\newreptheorem{prop}{Proposition}  

\theoremstyle{definition}
\newtheorem{deft}[teor]{Definition}

\theoremstyle{remark}
\newtheorem{oss}[teor]{Remark}

\DeclareMathOperator\sign{sign}

\DeclareMathOperator\bbC{\mathbb{C}}

\DeclareMathOperator\bbH{\mathbb{H}}

\DeclareMathOperator\bbP{\mathbb{P}}

\DeclareMathOperator\bbR{\mathbb{R}}

\DeclareMathOperator\calB{\mathcal{B}}

\DeclareMathOperator\calV{\mathcal{V}}

\DeclareMathOperator\psl{\textup{PSL}}

\DeclareMathOperator\vol{\textup{Vol}}

\DeclareMathOperator\scrF{\mathscr{F}}

\begin{document}

\title[Borel invariant for Zimmer cocycles]{Borel invariant for measurable cocycles of $3$-manifold groups}

\author[A. Savini]{A. Savini}
\address{Section de Math\'ematiques, University of Geneva, Rue du Li\`evre 2, 1227 Geneva, Switzerland}
\email{Alessio.Savini@unige.ch}

\date{\today.\ \copyright{\ A. Savini 2019, Partially supported by the project \emph{Geometric and harmonic analysis with applications}, funded by EU Horizon 2020 under the Marie Curie grant agreement No 777822.}.}

\begin{abstract}

We introduce the notion of pullback along a measurable cocycle and we use it to extend the Borel invariant studied by Bucher, Burger and Iozzi to the world of measurable cocycles. 
The Borel invariant is constant along cohomology classes and has bounded absolute value. This allows to define maximal cocycles. We conclude by proving that maximal cocycles are actually trivializable to the restriction of the irreducible representation. 
 
\end{abstract}
  
\maketitle

\section{Introduction}

Rigidity theory is a mathematical subject which has been widely studied so far. One of the most celebrated result is Mostow Rigidity Theorem \cite{mostow68:articolo,Most73} which states that the isomorphism class of a lattice inside a simple Lie group $G$ of non-compact type boils down to its conjugacy class, provided that $G$ is not isomorphic to $\psl(2,\bbR)$. For higher rank lattices an even more rigid behaviour occurs. Indeed, when the target is an algebraic group over a local field of characteristic zero, Margulis \cite{margulis:super} proved that Zariski dense representations of a higher rank lattice can be actually extended to the ambient group. This property, named \emph{superrigidity}, was later extended by Zimmer \cite{zimmer:annals} to measurable cocycles. 

One of the possible ways to approach rigidity is based on techniques coming from bounded cohomology theory. For instance, when $\Gamma \leq \textup{Isom}^+(\bbH^n_{\bbR})$ is a torsion-free lattice with $n \geq 3$, Bucher, Burger and Iozzi \cite{bucher2:articolo} introduced the notion of \emph{volume} for a representation $\rho:\Gamma \rightarrow \textup{Isom}^+(\bbH^n_{\bbR})$. They showed that this invariant has bounded absolute value and it attains its maximum if and only if the representation $\rho$ is conjugated to the standard lattice embedding. 

More surprisingly, a similar result can be stated also for representations into $\psl(n,\bbC)$, when $\Gamma$ is a torsion-free lattice of $\psl(2,\bbC)$. Indeed Bucher, Burger and Iozzi \cite{BBIborel} exploited the \emph{bounded Borel class} $\beta_b(n) \in \textup{H}^3_{cb}(\psl(n,\bbC);\bbR)$ to introduce the \emph{Borel invariant} $\beta_n(\rho)$, whose behaviour is very similar to the volume. Namely, its absolute value is bounded and it is maximal only on the conjugacy class of the irreducible representation restricted to $\Gamma$. 

The purpose of this paper is to develop a framework to define the same invariant for measurable cocycles. Let $\Gamma \leq \psl(2,\bbC)$ be a torsion-free lattice. Let $(X,\mu_X)$ be a standard Borel probability $\Gamma$-space. Consider a measurable cocycle $\sigma:\Gamma \times X \rightarrow \psl(n,\bbC)$. We first define a way to pullback bounded cohomology classes via the measurable cocycle $\sigma$ (Section \ref{sec:pullback:cocycle}). Using this procedure, we can follow the same line of \cite{BBIborel} to define the Borel invariant $\beta_n(\sigma)$ associated to $\sigma$.

\begin{teor}\label{teor:rigidity}
Let $\Gamma \leq \psl(2,\bbC)$ be a torsion-free lattice and consider a standard Borel probability $\Gamma$-space. For any measurable cocycle $\sigma:\Gamma \times X \rightarrow \psl(n,\bbC)$ it holds 
$$
|\beta_n(\sigma)| \leq {n+1 \choose 3}\textup{Vol}(\Gamma \backslash \bbH^3_{\bbR}) \ ,
$$
and the equality holds if and only if the cocycle is cohomologous either to the irreducible representation $\pi_n:\psl(2,\bbC) \rightarrow \psl(n,\bbC)$ restricted to $\Gamma$ or to its complex conjugated. 
\end{teor}

Clearly the above statement generalizes the rigidity theorem given by Bucher, Burger and Iozzi \cite{BBIborel}. Indeed any representation $\rho:\Gamma \rightarrow \psl(n,\bbC)$ determines canonically a cocycle $\sigma_\rho$ and its Borel invariant is equal to the one of the representation $\rho$ (Proposition \ref{prop:borel:representation}). 

The proof of Theorem \ref{teor:rigidity} is inspired by the proof by Bader, Furman and Sauer \cite{sauer:articolo} of the $1$-tautness of the group $\textup{Isom}(\bbH^n_{\bbR})$. In the proof we will exploit the existence of a \emph{boundary map}, namely a measurable map $\phi:\bbP^1(\bbC) \times X \rightarrow \scrF(n,\bbC)$ which is $\sigma$-equivariant. Here $\scrF(n,\bbC)$ is the space of complete flags in $\bbC^n$. The existence of such a map for maximal cocycles is proved in Proposition \ref{prop:existence:boundary}. Since we can implement the pullback using a boundary map, this leads to a formula similar to \cite[Equation 15]{BBIborel}.

\begin{prop}\label{prop:formula}
Let $\Gamma \leq \psl(2,\bbC)$ be a torsion-free lattice and consider a standard Borel probability $\Gamma$-space. Let $\sigma:\Gamma \times X \rightarrow \psl(n,\bbC)$ be a measurable cocycle with a boundary map $\phi:\bbP^1(\bbC) \times X \rightarrow \scrF(n,\bbC)$. Then for every $\xi_0,\xi_1,\xi_2,\xi_3 \in \bbP^1(\bbC)$ it holds
$$
\frac{\beta_n(\sigma)}{\textup{Vol}(\Gamma \backslash \bbH^3_{\bbR})} \textup{Vol}(\xi_0,\ldots,\xi_3)=\int_{\Gamma \backslash \psl(2,\bbC)} \int_X B_n(\phi(\overline{g}\xi_0,x),\ldots,\phi(\overline{g}\xi_3,x))d\mu_X(x)d\mu(\overline{g}) \ ,
$$
where $\textup{Vol}$ is the hyperbolic volume function and $B_n$ is the Borel function. Here $\mu$ is the normalized probability measure induced on $\Gamma \backslash \psl(2,\bbC)$ induced by the Haar measure. 
\end{prop}

Thanks to the proposition above, under the assumption of maximality, the slice $\phi_x(\ \cdot \ )=\phi(\ \cdot, x)$ is measurable and maps almost every maximal ideal tetrahedron to a maximal configuration of flags, for almost every $x \in X$. Since such map lies in the $\psl(n,\bbC)$-orbit of the Veronese embedding $\calV_n:\bbP^1(\bbC) \rightarrow \scrF(n,\bbC)$, we get the measurable map $f$ which realizes the desired trivialization to the irreducible representation. 

As discussed in \cite{savini:tetr}, parabolic representations cannot be maximal. We are going to prove a similar statement also in this context. 

\begin{prop}\label{prop:parabolic:cocycle}
Let $\Gamma \leq \psl(2,\bbC)$ be a torsion-free lattice and consider a standard Borel probability $\Gamma$-space. Let $\sigma:\Gamma \times X \rightarrow \psl(n,\bbC)$ be a measurable cocycle. Suppose that $\sigma$ is cohomologous to a cocycle taking values into a parabolic subgroup $P$ of $\psl(n,\bbC)$. If $(n_1,\ldots,n_r)$ is the partition of $n$ associated to the subgroup $P$ then 
$$
|\beta_n(\sigma)| \leq \sum_{i=1}^r {n_i+1 \choose 3}\textup{Vol}(\Gamma \backslash \bbH^3_{\bbR}) \ .
$$
In particular $\sigma$ cannot be maximal. 
\end{prop}

Notice that the machinery developed here is the first instance of a more general framework studied by the author, Moraschini and Sarti \cite{moraschini:savini,moraschini:savini:2,savini:surface,sarti:savini,sarti:savini:2}. Indeed the Borel invariant is the first example of \emph{multiplicative constant} and Proposition \ref{prop:formula} is the first example of \emph{multiplicative formula} introduced and studied by the author.  

\subsection*{Plan of the paper}

The paper is organized as follows. The first section is dedicated to preliminary definitions. We recall the notion of measurable cocycle, then we move to the notion of continuous bounded cohomology for locally compact groups. Finally we conclude the section reminding the definition of the Borel invariant and we state its rigidity property.

The second section is devoted to the definition of pullback along a measurable cocycle. Using this notion we can define the Borel invariant of a measurable cocycle and we study some of its properties. We conclude by showing the main rigidity theorem. 

\subsection*{Acknowledgements} I would like to thank Marco Moraschini for having suggested my the paper by Bader, Furman and Sauer \cite{sauer:articolo}.

\section{Preliminary definitions}

\subsection{Basic aspects of measurable cocycles theory}

We briefly report here the main definitions about measurable cocycles that we will need in the sequel. We refer the reader either to Furstenberg's work \cite{furst:articolo} or to Zimmer's book \cite{zimmer:libro}. 

Let $G,H$ be two locally compact second countable groups endowed with their Haar measurable structures. We consider a standard Borel probability space $(X,\mu_X)$ with an essentially free measure-preserving $G$-action. We are going to call such a space a \emph{standard Borel probability $G$-space}. If $(Y,\nu)$ is another measure space, we are going to denote by $\textup{Meas}(X,Y)$ the space of measurable functions identified up to zero measure sets. We endow $\textup{Meas}(X,Y)$ with the topology of the convergence in measure. 

\begin{deft}\label{def:measurable:cocycle}
A measurable function $\sigma:G \times X \rightarrow H$ is a \emph{measurable cocycle} if it satisfies 
\begin{equation}\label{eq:cocycle}
\sigma(g_1g_2,x)=\sigma(g_1,g_2x)\sigma(g_2,x) \ ,
\end{equation}
for almost every $g_1,g_2 \in G$ and almost every $x \in X$. Here $g_2x$ denotes the $G$-action on $X$. 
\end{deft}

Looking at $\sigma$ as an element of $\textup{Meas}(G,\textup{Meas}(X,H))$, one can verify that Equation \eqref{eq:cocycle} boils down to the equation of Borel $1$-cocycle in the standard Eilenberg-MacLance cohomology (see Feldman and Moore \cite{feldman:moore}). Such an interpretation leads naturally to the notion of cohomology class for a measurable cocycle. 

\begin{deft}\label{def:cohomology:class}
Let $\sigma_1,\sigma_2:G \times X \rightarrow H$ be two measurable cocycles. Consider a measurable map $f:X \rightarrow H$. The \emph{$f$-twisted cocycle associated to $\sigma_1$} is the map 
$$
\sigma_1^f:G \times X \rightarrow H \ , \ \ \ \sigma_1^f(g,x):=f(gx)^{-1}\sigma(g,x)f(x) \ ,
$$
for almost every $g \in G, x \in X$. 

We say that $\sigma_1,\sigma_2$ are \emph{cohomologous} if there exists a measurable map $f:X \rightarrow H$ such that 
$$
\sigma_2=\sigma_1^f \ .
$$
\end{deft}

Measurable cocycles are quite ubiquitous in Mathematics. Representations are a particular example of measurable cocycles. Indeed, given any continuous representation $\rho:G \rightarrow H$ we can define the map
$$
\sigma_\rho:G \times X \rightarrow H \ , \ \ \ \sigma_\rho(g,x):=\rho(g) \ ,
$$
for every $g \in G$ and almost every $x \in X$. The fact that $\rho$ is a morphism implies easily that $\sigma_\rho$ is a measurable cocycle called \emph{measurable cocycle associated to $\rho$}. Notice that in the particular case when $G$ is a lattice, any representation is automatically continuous and hence it defines a measurable cocycle once we have fixed the standard Borel probability space $(X,\mu_X)$. 

We conclude this short exposition about measurable cocycles with the notion of boundary map. We are going to suppose that both $G$ and $H$ are two semisimple Lie groups of non-compact type (or lattices in such Lie groups). Denote by $B(G)$ the Furstenberg-Poisson boundary associated to $G$, which can be identified with the quotient $G/P$, where $P \leq G$ is a \emph{minimal parabolic subgroup}. Consider a measure space $(Y,\nu)$ such that $H$ acts on $Y$ by preserving the measure class of $\nu$. 

\begin{deft}\label{def:boundary:map}
Let $\sigma:G \times X \rightarrow H$ be a measurable cocycle. A \emph{boundary map} $\phi:B(G) \times X \rightarrow Y$ is a measurable map which is $\sigma$-\emph{equivariant}, namely
$$
\phi(g\xi,gx)=\sigma(g,x)\phi(\xi,x) \ ,
$$
for almost every $\xi \in B(G), x \in X$. 
\end{deft}

The existence and the uniqueness of boundary maps relies on dynamical properties of the cocycle $\sigma$, as shown for instance by Furstenberg \cite{furst:articolo}. For \emph{non-elementary} cocycles taking values into the isometry group of a negatively curved space, Monod and Shalom \cite[Proposition 3.3]{MonShal0} proved that a boundary map exists and it is essentially unique. We are going to prove the existence of a boundary map for maximal cocycles in Proposition \ref{prop:existence:boundary}. 

\subsection{Continuous bounded cohomology}

The following section is devoted to the definitions and the properties of both continuous cohomology and continuous bounded cohomology. For a more detailed discussion about those topics we refer the reader to the work of Burger and Monod \cite{monod:libro,burger2:articolo}. 

Let $G$ be a locally compact second countable group. A \emph{Banach $G$-module} $(E,\pi)$ is a Banach space $E$ together with a linear isometric action $\pi:G \rightarrow \textup{Isom}(E)$. We are going to denote by $E^G$ the subspace of \emph{$G$-invariant} vectors, that is those vectors $v \in E$ such that $\pi(g)v=v$ for every $g \in G$. 

Given a Banach $G$-module $E$, the space of \emph{continuous $E$-valued functions} is given by 
$$
\textup{C}^\bullet_{c}(G;E):= \{ f:G^{\bullet+1} \rightarrow E \ | \ f  \ \text{ is continuous} \} \ , 
$$
and we endow it with the following isometric action 
$$
(gf)(g_0,\ldots,g_\bullet):=\pi(g)f(g^{-1}g_0,\ldots,g^{-1}g_\bullet) \ ,
$$
for every $f \in \textup{C}_c^\bullet(G;E)$ and every $g,g_0,\ldots,g_\bullet \in G$. Together with the \emph{standard homogeneous coboundary operator} 
$$
\delta^\bullet:\textup{C}_c^\bullet(G;E) \rightarrow \textup{C}_c^{\bullet+1}(G;E) \ ,
$$
$$
\delta^\bullet f(g_0,\ldots,g_{\bullet+1}):=\sum_{i=0}^{\bullet+1} (-1)^i f(g_0,\ldots,g_{i-1},g_{i+1},\ldots,g_{\bullet+1}) \ ,
$$
we obtain an exact complex $(\textup{C}^\bullet_c(G;E),\delta^\bullet)$. 

\begin{deft}
The \emph{continuous cohomology} of $G$ with coefficients in $E$ is the cohomology of the complex $(\textup{C}^\bullet_c(G;E)^G,\delta^\bullet)$. It is denoted by $\textup{H}^\bullet_c(G;E)$. 
\end{deft}

The space $\textup{C}^\bullet_c(G;E)$ has a natural normed structure. Indeed for every $f \in \textup{C}^\bullet_c(G;E)$ we can define
$$
\lVert f \rVert_\infty:=\sup\{ \lVert f(g_0,\ldots,g_\bullet) \rVert_E \ | \ g_0,\ldots,g_\bullet \in G \} \ ,
$$
where $\lVert \  \cdot \ \rVert_E$ is the norm on $E$. A continuous $E$-valued function is \emph{bounded} if its norm is finite. We denote by $\textup{C}^\bullet_{cb}(G;E)$ the submodule of continuous bounded $E$-valued functions. Since $\delta^\bullet$ preserves boundedness, we can restrict it to the space $\textup{C}^\bullet_{cb}(G;E)^G$. 

\begin{deft}
The \emph{continuous bounded cohomology} of $G$ with coefficients in $E$ is the cohomology of the complex $(\textup{C}^\bullet_{cb}(G;E)^G,\delta^\bullet)$. It is denoted by $\textup{H}^\bullet_{cb}(G;E)$. 
\end{deft}

Notice that the $\textup{L}^\infty$-norm on cochains induces a natural quotient seminorm in cohomology. We say that an isomorphism between seminormed cohomology groups is \emph{isometric} if the corresponding quotient seminorms are preserved. 

A way to study the differences between continuous cohomology and continuous bounded cohomology relies on the \emph{comparison map}
$$
\textup{comp}^\bullet_G:\textup{H}^\bullet_{cb}(G;E) \rightarrow \textup{H}^\bullet_{c}(G;E) \ ,
$$
which is the map induced in cohomology by the inclusion 
$$
i:\textup{C}^\bullet_{cb}(G;E)^G \rightarrow \textup{C}^\bullet_{c}(G;E)^G \ .
$$

The computation of continuous bounded cohomology may be quite difficult if one want to follow only the definition we gave. Burger and Monod \cite{monod:libro,burger2:articolo} circumvented this problem adopting the use of strong resolutions by relatively injective $G$-modules. Since it would be too technical to introduce such notions here, we prefer to omit them and we refer the reader to \cite{monod:libro} for a broad discussion. 

Here we are going to construct explicitly only one resolution that we are going to use several times in the paper. Suppose that the Banach $G$-module $(E,\pi)$ is the dual of some Banach space. In that case we can endow $E$ with the weak-$^*$ topology and the weak-$^*$ measurable structure. Suppose that $G$ is a simple Lie group of non-compact type and denote by $B(G)$ the Furstenberg-Poisson boundary of $G$. 

Let
\begin{align*}
\mathcal{B}^\infty(B(G)^{\bullet+1};E):=\{ f:B(G)^{\bullet+1} \rightarrow E \ &| \ f \ \text{is weak-$^*$ measurable}\\
\sup_{\xi_0,\ldots,\xi_\bullet}& \lVert  f(\xi_0,\ldots,\xi_\bullet) \rVert_E \} \ . 
\end{align*}
be the space of \emph{bounded weak-$^*$ measurable $E$-valued functions}. Similarly, the space of \emph{essentially bounded weak-$^*$ measurable $E$-valued functions} is given by
$$
\textup{L}^\infty_{\textup{w}^*}(B(G)^{\bullet+1};E):=\{ \ [f] \ | \ f \in \mathcal{B}^\infty(B(G)^{\bullet+1};E) \ \} \ ,
$$
where $[f]$ is the equivalence class of functions identified up to a zero measure set. With an abuse of notation we are going to drop the parenthesis referring only to a representative of the class. A function will be \emph{alternating} if 
$$
\sign(\varepsilon)f(\xi_0,\ldots,\xi_\bullet)=f(\xi_{\varepsilon(0)},\ldots,\xi_{\varepsilon(\bullet)}) \ ,
$$
where $\varepsilon \in S_{\bullet+1}$ is a permutation and $\sign(\varepsilon)$ denotes its parity. 

We endow the function space $\mathcal{B}^\infty(B(G)^{\bullet+1};E)$ (respectively $\textup{L}^\infty_{\textup{w}^\ast}(B(G)^{\bullet+1};E)$)  with the $G$-action 
$$
(gf)(\xi_0,\ldots,\xi_\bullet):=\pi(g)f(g^{-1}\xi_0,\ldots,g^{-1}\xi_\bullet) \ ,
$$
where $g \in G, \xi_0,\ldots,\xi_\bullet \in B(G)$ and finally $f \in \mathcal{B}^\infty(B(G)^{\bullet+1};E)$ (respectively $f \in \textup{L}^\infty_{\textup{w}^\ast}(B(G)^{\bullet+1};E)$). 

Together with the standard homogeneous coboundary operator, we can construct a complex $(\textup{L}^\infty_{\textup{w}^\ast}(B(G)^{\bullet+1};E),\delta^\bullet)$. The main result concerning such complex is by Monod.

\begin{teor}\cite[Theorem 7.5.3]{monod:libro}\label{teor:resolution:boundary}
Let $G$ be a locally compact second countable group and let $(E,\pi)$ be a Banach $G$-module which is the dual of some Banach space. The cohomology of the complex $(\textup{L}^\infty_{\textup{w}^\ast}(B(G)^{\bullet+1};E)^G,\delta^\bullet)$ computes isometrically the continuous bounded cohomology $\textup{H}^\bullet_{cb}(G;E)$. 

The same result holds also for the subcomplex of alternating functions.
\end{teor}

In some cases it can be useful to work directly with $\mathcal{B}^\infty(B(G)^{\bullet+1};E)$. Indeed we have the following

\begin{prop}\cite[Proposition 2.1,Corollary 2.2]{burger:articolo}\label{prop:strong:resolution}
Let $G$ be a locally compact second countable group and let $(E,\pi)$ be a Banach $G$-module which is the dual of some Banach space. The complex $(\mathcal{B}^\infty(B(G)^{\bullet+1};E),\delta^\bullet)$ is a strong resolution of $E$. As a consequence there exists a canonical map 
$$
\mathfrak{c}^n:\textup{H}^n(\mathcal{B}^\infty(B(G)^{\bullet+1};\bbR)^G) \rightarrow \textup{H}^n(\textup{L}^\infty_{\textup{w}^\ast}(B(G)^{\bullet+1};E)^G) \cong \textup{H}^n_{cb}(G;E) \ ,
$$
for every $n \geq 0$. 
\end{prop}

\subsection{Transfer maps}
In this section we are going to recall two maps which \emph{transfer} the information from the lattice to the ambient group. We will use those maps to show that the Borel invariant is a \emph{multiplicative constant} in the sense of \cite[Definition 3.20]{moraschini:savini:2}. 

Let $\Gamma \leq \psl(2,\bbC)$  be a torsion-free lattice. We can define the following map at the level of cochains.
$$
\widehat{\textup{trans}}^\bullet_\Gamma:\textup{L}^\infty(\bbP^1(\bbC)^{\bullet+1};\bbR)^\Gamma \rightarrow \textup{L}^\infty(\bbP^1(\bbC);\bbR)^{\psl(2,\bbC)} \ ,
$$
$$
\widehat{\textup{trans}}^\bullet_\Gamma(\psi)(\xi_0,\ldots,\xi_\bullet):=\int_{\Gamma \backslash \psl(2,\bbC)} \psi(\overline{g}\xi_0,\ldots,\overline{g}\xi_\bullet)d\mu(\overline{g}) \  ,
$$
where $\mu$ is the normalized probability measure on $\Gamma \backslash \psl(2,\bbC)$ induced by the Haar measure on $G$. 

Since $\widehat{\textup{trans}}^\bullet_\Gamma$ is linear, it is easy to check that is commutes with the coboundary operator. Additionally if $\psi$ is $\Gamma$-invariant, its image is $\psl(2,\bbC)$-invariant by the $\psl(2,\bbC)$-invariance of the measure $\mu$.  

\begin{deft}	\label{def:transfer:map}
The \emph{transfer map} is the map induced in bounded cohomology by $\widehat{\textup{trans}}^\bullet_\Gamma$, that is 
$$
\textup{trans}^\bullet_\Gamma:\textup{H}^\bullet_b(\Gamma;\bbR) \rightarrow \textup{H}^\bullet_{cb}(\psl(2,\bbC);\bbR) \ . 
$$
\end{deft}

The second map we want to introduce require more work. Denote by $M=\Gamma \backslash \bbH^3_{\bbR}$. If the manifold is not compact we are going to consider a compact core $N \subset M$. Recall that to compute the cohomology $\textup{H}^\bullet(M,M \setminus N;\bbR)$ we can exploit the De Rham isomorphism and use the complex of differential forms $\Omega^\bullet(M,M \setminus N;\bbR)$ defined on $M$ and vanishing on $M \setminus N$ (in the case that $M$ is compact the latter space is empty). 

Following \cite{BBIborel}, we denote by $U=p^{-1}(M \setminus N)$ the preimage under the covering map $p:\bbH^3_{\bbR} \rightarrow M$ of a finite union of horocyclic neighborhoods of the cusps (again we stress that if $M$ is compact the latter space is empty). Since we can indentify the space $\Omega^\bullet(M,M\setminus N;\bbR)$ with the $\Gamma$-invariant differential forms $\Omega^\bullet(\bbH^3_{\bbR},U;\bbR)^\Gamma$ on $\bbH^3_{\bbR}$ vanishing on $U$, we can define 
$$
\widehat{\tau}_{\textup{dR}}^\bullet:\Omega^\bullet(\bbH^3_{\bbR},U;\bbR)^\Gamma \rightarrow \Omega^\bullet(\bbH^3_{\bbR};\bbR)^{\psl(2,\bbC)} \ ,
$$
$$
\widehat{\tau}_{\textup{dR}}^\bullet(\psi):=\int_{\Gamma \backslash \psl(2,\bbC)} \overline{g}^\ast \psi d\mu(\overline{g})
$$
where $\overline{g}^\ast \psi$ denotes the pullback form and $\mu$ is the usual normalized probability measure on $\Gamma \backslash \psl(2,\bbC)$. 

One can apply the same arguments valid for $\widehat{\textup{trans}}^\bullet_{\Gamma}$ to obtain a well-defined map in cohomology.

\begin{deft}
We denote by $\tau^\bullet_{\textup{dR}}$ the map induced in cohomology by $\widehat{\tau}_{\textup{dR}}^\bullet$, that is
$$
\tau^\bullet_{\textup{dR}}:\textup{H}^\bullet(N,\partial N;\bbR) \cong \textup{H}^\bullet(M,M \setminus N;\bbR) \rightarrow \textup{H}^\bullet(\Omega^\bullet(\bbH^3_{\bbR};R)^{\psl(2,\bbC)}) \cong \textup{H}^\bullet_c(\psl(2,\bbC);\bbR) \ .
$$
The isomorphism appearing on the left is due to the invariance of cohomology with respect to homotopy equivalence. The isomorphism on the right is the Van Est isomorphism \cite[Cor. 7.2]{guichardet}.
\end{deft}

\subsection{The Borel cocycle} 
We are going to recall briefly the definition and the properties of the \emph{Borel cocycle}. We refer the reader to \cite{BBIborel},\cite{savini:tetr} for more details.  

Define the set
\[
\mathfrak{S}_k(m):=\{ (x_0,\ldots,x_k) \in (\bbC^m)^{k+1} | \langle x_0,\ldots x_k \rangle = \bbC^m \} / \textup{GL}(m,\bbC)
\]
where $\textup{GL}(m,\bbC)$ acts on $(k+1)$-tuples of vectors by the diagonal action and $\langle x_0,\ldots x_k \rangle$ is the $\bbC$-linear space generated by $x_0,\ldots,x_k$. When $k < m-1$ the space defined above is obviously empty. If $V$ is an $m$-dimensional vector space over $\bbC$, given any $(k+1)$-tuple of spanning vectors $(x_0,\ldots,x_k) \in V^{k+1}$, we can fix an isomorphism $V \rightarrow \bbC^m$. Any two different isomorphisms are related by an element $g \in \textup{GL}(m,\bbC)$ and hence they both determine a well-defined element of $\mathfrak{S}_k(m)$ which will be denoted by $[V;(x_0,\ldots,x_k)]$. We set
\[
\mathfrak{S}_k:= \bigsqcup_{m \geq 0} \mathfrak{S}_k(m)= \mathfrak{S}_k(0) \sqcup \ldots \sqcup \mathfrak{S}_k(k+1).
\]

The hyperbolic volume function $\vol:(\bbP^1(\bbC))^4 \rightarrow \bbR$ can be thought of as defined on $(\bbC^2 \setminus \{0\})^4$, so it is extendable to 
\[
\vol: \mathfrak{S}_3 \rightarrow \bbR
\]
where we set $\vol|\mathfrak{S}_3(m)$ to be identically zero if $m \neq 2$ and 
\[
\vol[\bbC^2;(v_0,\ldots,v_3)]:=
	\begin{cases}
	&\vol(v_0,\ldots,v_3) \hspace{10pt} \text{if each $v_i \neq 0$,}\\
	&0 \hspace{10pt} \text{otherwise}.
	\end{cases}
\]
 
The function above is the key tool that we use to define a cocycle on the space $\mathscr{F}_\textup{aff}(n,\bbC)^4$ of $4$-tuples of affine flags. A \textit{complete flag} $F$ of $\bbC^n$ is a sequence of nested subspaces
\[
F^0 \subset F^1 \subset  \ldots F^{n-1} \subset F^n
\]
where $\dim_{\bbC} F^i=i$ for $i=1,\ldots,n$. The space $\scrF(n,\bbC)$ of all the possible complete flags of $\bbC^n$ is a complex variety which can be thought of as the quotient of $\psl(n,\bbC)$ by any of its Borel subgroups. An \textit{affine flag} $(F,v)$ of $\bbC^n$ is a complete flag $F$ with the additional datum of a decoration $v=(v^1,\ldots,v^n) \in (\bbC^n)^n$ such that
\[
F^i=\bbC v^i+F^{i-1}
\]
for every $i=1,\ldots n$. Given any $4$-tuple of affine flags $\mathbf{F}=((F_0,v_0),\ldots,(F_3,v_3))$ of $\bbC^n$ and a multi-index $\mathbf{J} \in \{ 0,\ldots,n-1\}^4$, we set

\[
\mathcal{Q}(\mathbf{F},\mathbf{J}):=\left[ \frac{\langle F_0^{j_0+1},\ldots,F_3^{j_3+1} \rangle}{\langle F_0^{j_0},\ldots,F_3^{j_3} \rangle};(v_0^{j_0+1},\ldots,v_3^{j_3+1})\right],
\]
which is an element of $\mathfrak{S}_3$. With the previous notation, we define the cocycle $B_n$ as 
\[
B_n((F_0,v_0),\ldots,(F_3,v_3)):=\sum_{\mathbf{J} \in \{0,\ldots,n-1\}^4} \vol \mathcal{Q}(\mathbf{F},\mathbf{J}) \ .
\]

Bucher, Burger and Iozzi proved the following 

\begin{prop}\cite[Corollary 13, Theorem 14]{BBIborel}\label{prop:borel:cocycle:BBI}
The function $B_n$ does not depend on the decoration used to compute it and hence it descends naturally to a function 
$$
B_n:\scrF(n,\bbC)^4 \rightarrow \bbR \ ,
$$
on $4$-tuples of flags which is defined everywhere. Moreover that function is a measurable $\psl(n,\bbC)$-invariant alternating cocycle whose absolute value is bounded by ${n+1 \choose 3}\nu_3$, where $\nu_3$ is the volume of a positively oriented regular ideal tetrahedron in $\bbH^3_{\bbR}$.
\end{prop}

As a consequence of Proposition \ref{prop:strong:resolution} the function $B_n$ determines a bounded cohomology class in $\textup{H}^3_{cb}(\psl(n,\bbC);\bbR)$. We are going to denote such a class by $\beta_b(n)$. 

\begin{deft}
The cocycle $B_n$ is called \emph{Borel cocycle} and the class $\beta_b(n)$ is called \emph{bounded Borel class}. 
\end{deft}

Generalizing a previous result by Bloch~\cite{bloch:libro} for $\psl(2,\bbC)$, in \cite[Theorem 2]{BBIborel} the authors proved that the cohomology group $\textup{H}^3_{cb}(\psl(n,\bbC);\bbR)$ is a one-dimensional vector space generated by the bounded Borel class.

Since we are going to use it later, we recall the main rigidity property of the Borel cocycle. Denote by $\calV_n:\bbP^1(\bbC) \rightarrow \scrF(n,\bbC)$ the \emph{Veronese embedding} of the complex projective line into the space of complete flags. It is defined as follows. Let $\calV_n(\xi)_i$ be the $i$-dimensional space of the flag $\calV_n(\xi)$. If $\xi$ has homogeneous coordinates $[x:y]$, the $(n-i)$-dimensional subspace $\calV_n^{n-i}(\xi)$ has a basis given by
\[
\left( 0, \ldots, 0, x^i, {{i}\choose{1}}x^{i-1}y,\ldots, {{i}\choose{j}}x^{i-j}y^j, \ldots, {{i}\choose{i-1}}xy^{i-1},y^i,0,\ldots,0 \right)^T
\]
where the first are $k$ zeros and the last are $n-i-k-1$ zeros, for $k=0,\ldots,n-1-i$. 

\begin{deft}
Let $(F_0,\ldots,F_3) \in \scrF(n,\bbC)^4$ be a $4$-tuple of flags. We say that the $4$-tuple is \textit{maximal} if
\[
|B_n(F_0,\ldots,F_3)|={{n+1}\choose{3}}\nu_3.
\]
\end{deft}

The Veronese embedding allows to describe completely the set of maximal flags. More precisely,~\cite[Theorem 19]{BBIborel} shows that if a $4$-tuple of flags $(F_0,\ldots,F_3)$ is maximal, then there must exists an element $g \in \psl(n,\bbC)$ such that
\[
g(F_0,F_1,F_2,F_3)=(\calV_n(0),\calV_n(1).\calV_n(\pm e^{\frac{i \pi}{3}}),\calV_n(\infty)) \  ,
\]
where the sign $\pm$ reflects the sign of $B_n(F_0,\ldots,F_3)$. Additionally if $B_n(F_0,F_1,F_2,F_3)=B_n(F_0,F_1,F_2,F_3')=\pm {n+1 \choose 3} \nu_3$, then $F_3=F_3'$ \cite[Corollary 20]{BBIborel}.

More generally, this rigidity result can be extended to the context of measurable maps

\begin{prop}\cite[Proposition 29]{BBIborel}\label{prop:maximal:map}
Let $\varphi: \bbP^1(\bbC) \rightarrow \scrF(n,\bbC)$ be a measurable map. Suppose that for almost every regular ideal tetrahedron $(\xi_0,\ldots,\xi_3) \in (\bbP^1(\bbC))^4$, the $4$-tuple $(\varphi(\xi_0),\ldots,\varphi(\xi_3))$ is a maximal configuration of flags. Then there must exist an element $g \in \psl(n,\bbC)$ such that $$g\varphi(\xi)=\calV_n(\xi) \ ,$$ for almost every $\xi \in \bbP^1(\bbC)$. 
\end{prop}

We will heavily exploit the previous statement in the proof of the main theorem to construct explicitly the measurable map which trivializes a maximal measurable cocycle. 

\section{Pullback along a measurable cocycle}\label{sec:pullback:cocycle}

In this section we are going to introduce the main ingredient for defining the Borel invariant of a measurable cocycle. We will show how a measurable cocycle naturally induces a map in bounded cohomology and we are going to implement the pullback using boundary maps.

Let $\Gamma \leq G$ be a torsion-free lattice in a simple Lie group and let $(X,\mu_X)$ be a standard Borel probability $\Gamma$-space. Consider a locally compact second countable group $H$. 

Given a measurable cocycle $\sigma:\Gamma \times X \rightarrow H$, we can define the following map
$$
\textup{C}^\bullet_b(\sigma):\textup{C}_{cb}^\bullet(H;\bbR) \rightarrow \textup{C}_b^\bullet(\Gamma;\bbR) \ ,
$$
$$
\psi \mapsto \textup{C}^\bullet_b(\sigma)(\psi)(\gamma_0,\ldots,\gamma_\bullet):=\int_X \psi(\sigma(\gamma_0^{-1},x)^{-1},\ldots,\sigma(\gamma_\bullet^{-1},x)^{-1})d\mu_X(x) \ .
$$
The definition given above may appear quite strange, but it is actually inspired by \cite[Theorem 5.6]{sauer:companion} and by the induction map associated to a coupling defined by Monod and Shalom \cite{MonShal}. 

\begin{lem}\label{lemma:pullback:cocycle}
Let $\sigma:\Gamma \times X \rightarrow H$ be a measurable cocycle. The map $\textup{C}^\bullet_b(\sigma)$ is a well-defined cochain map which induces a map on the subcomplexes of invariant vectors, that is
$$
\textup{C}^\bullet_b(\sigma):\textup{C}_{cb}^\bullet(H:\bbR)^H \rightarrow \textup{C}_b(\Gamma;\bbR)^\Gamma \ .
$$
As consequence it induces a map at the level of bounded cohomology groups, namely
$$
\textup{H}^\bullet_{b}(\sigma):\textup{H}^\bullet_{cb}(H;\bbR) \rightarrow \textup{H}^\bullet_b(\Gamma;\bbR) \ , \ \ \textup{H}^\bullet_b(\sigma)([\psi]):=\left[ \textup{C}^\bullet_{cb}(\sigma)(\psi) \right] \ .
$$
\end{lem}

\begin{proof}
It is immediate to verify that $\textup{C}^\bullet_{b}(\sigma)$ preserves boundedness, since $\mu_X$ is a probability measure. The fact that it is a cochain map is an easy computation that we leave to the reader. 

We are left to show that given a $H$-invariant continuous cochain, its image is $\Gamma$-invariant. Let $\psi \in \textup{C}^\bullet_{cb}(H;\bbR)$ and consider $\gamma,\gamma_0,\ldots,\gamma_\bullet \in \Gamma$. 

\begin{align*}
\gamma \cdot \textup{C}^\bullet_b(\sigma)(\psi)(\gamma_0,\ldots,\gamma_\bullet) &=\int_X \psi( \sigma(\gamma_0^{-1} \gamma,x)^{-1},\ldots,\sigma(\gamma_\bullet^{-1}\gamma,x)^{-1})d\mu_X(x)=\\
&=\int_X \psi(\sigma(\gamma,x)^{-1}\sigma(\gamma_0^{-1},\gamma x),\ldots )d\mu_X(x)=\\
&=\int_X \psi(\sigma(\gamma_0^{-1},x),\ldots,\sigma(\gamma_\bullet^{-1},x))d\mu_X(x)=\\
&=\textup{C}^\bullet_b(\sigma)(\psi)(\gamma_0,\ldots,\gamma_\bullet) \ .
\end{align*}
In the computation we exploited Equation \eqref{eq:cocycle} to pass from the first line to the second one. Then we used the $\Gamma$-invariance of $\psi$ and $\mu$ to move from the second line to the third one. This concludes the proof. 
\end{proof}

Thanks to the previous lemma we can give the following

\begin{deft}
Let $\sigma:\Gamma \times X \rightarrow H$ be a measurable cocycle. The map 
$$
\textup{H}^\bullet_{b}(\sigma):\textup{H}^\bullet_{cb}(H;\bbR) \rightarrow \textup{H}^\bullet_b(\Gamma;\bbR)
$$
is the \emph{cohomological pullback induced by $\sigma$}. 
\end{deft}

The first thing that we want to show is that the pullback along a measurable cocycle actually generalizes the pullback along a representation. More precisely we have the following

\begin{lem}\label{lem:pullback:representation}
Let $\sigma:\Gamma \rightarrow H$ be a representation and let $(X,\mu_X)$ be any standard Borel probability $\Gamma$-space. If $\sigma_\rho:\Gamma \times X \rightarrow H$ is the measurable cocycle associated to $\rho$, then it holds
$$
\textup{H}^\bullet_b(\sigma_\rho)=\textup{H}^\bullet_b(\rho) \ .
$$
\end{lem}

\begin{proof}
We are going to verify that 
$$
\textup{C}^\bullet_b(\sigma_\rho)=\textup{C}^\bullet_b(\rho) \ .
$$
Let $\psi \in \textup{C}^\bullet_{cb}(H;\bbR)$ and consider $\gamma_0,\ldots,\gamma_\bullet \in \Gamma$. We have 
\begin{align*}
\textup{C}^\bullet_b(\sigma_\rho)(\psi)(\gamma_0,\ldots,\gamma_\bullet)&=\int_X \psi(\sigma_\rho(\gamma_0^{-1},x)^{-1},\ldots,\sigma_\rho(\gamma_\bullet^{-1},x)^{-1})d\mu_X(x)=\\
&=\int_X \psi(\rho(\gamma_0),\ldots,\rho(\gamma_\bullet))d\mu_X(x)=\\
&=\psi(\rho(\gamma_0),\ldots,\rho(\gamma_\bullet))=\textup{C}^\bullet_b(\rho)(\psi)(\gamma_0,\ldots,\gamma_\bullet) \ .
\end{align*}
In the computation we first used the definition of the cocycle $\sigma_\rho$ to move from the first line to the second one and then we concluded exploiting the fact that $\mu_X$ is a probability measure. The statement now follows. 
\end{proof}

A question which might arise quite naturally could be how the cohomological pullback varies along a the $H$-cohomology class of the cocycle $\sigma$. We are going to show that it is actually constant. 

\begin{lem}\label{lem:pullback:cohomology}
Let $\sigma:\Gamma \times X \rightarrow H$ be a measurable cocycle. Then for any measurable map $f:X \rightarrow H$ it holds
$$
\textup{H}^\bullet_b(\sigma^f)=\textup{H}^\bullet_b(\sigma) \ .
$$
Here $\sigma^f$ is the $f$-twisted cocycle associated to $\sigma$. 
\end{lem}

\begin{proof}
The proof will follow the line of \cite[Lemma 8.7.2]{monod:libro}. We want to show that there exists a chain homotopy between $\textup{C}^\bullet_b(\sigma)$ and $\textup{C}^\bullet_b(\sigma^f)$. 

Fix a continuous cochain $\psi \in \textup{C}_{cb}^\bullet(H;\bbR)^H$. We have
\begin{align}\label{eq:cochain:sigmaf}
\textup{C}^\bullet_{cb}(\sigma^f)(\psi)(\gamma_0,\ldots,\gamma_\bullet)&=\int_\Omega \psi(\sigma^f(\gamma_0^{-1},s)^{-1},\ldots,\sigma^f(\gamma_\bullet^{-1},s)^{-1})d\mu_\Omega=\\
&=\int_\Omega \psi(f(s)^{-1}\sigma(\gamma_0^{-1},s)^{-1}f(\gamma_0^{-1}s),\dots )d\mu_\Omega = \nonumber \\
&=\int_\Omega \psi(\sigma(\gamma_0^{-1},s)^{-1}f(\gamma_0^{-1}s),\ldots )d\mu_\Omega \nonumber \ .
\end{align}
Here we used the definition of $\sigma^f$ to move from the first line to the second one and then we exploited the $H$-invariance of $\psi$.  

For $0 \leq i \leq \bullet-1$ we now define the following map
$$
s^\bullet_i(\sigma,f):\textup{C}^\bullet_{cb}(H;\bbR) \rightarrow \textup{C}^{\bullet - 1}_b(\Gamma;\bbR) \ , \ \ s^\bullet_i(\sigma,f)(\psi)(\gamma_0,\ldots,\gamma_{\bullet-1}):=
$$
$$
=\int_\Omega \psi(\sigma(\gamma_0^{-1},s)^{-1}f(\gamma_0^{-1}s),\ldots,\sigma(\gamma_i^{-1},s)^{-1}f(\gamma_i^{-1}s),\sigma(\gamma_i^{-1},s)^{-1},\ldots,\sigma(\gamma_{\bullet-1}^{-1},s)^{-1})d\mu_\Omega(s) \ ,
$$
and we set $s^\bullet(\sigma,f):=\sum_{i=0}^{\bullet-1}(-1)^i s^\bullet_i(\sigma,f)$. By defining for $-1 \leq i \leq \bullet$ the map
$$
\rho^\bullet_i(\sigma,f):\textup{C}^\bullet_{cb}(H;\bbR) \rightarrow \textup{C}^{\bullet}_b(\Gamma;\bbR) \ , \ \ 
\rho^\bullet_i(\sigma,f)(\psi)(\gamma_0,\ldots,\gamma_\bullet):=
$$
$$
=\int_\Omega \psi(\sigma(\gamma_0^{-1},s)^{-1}f(\gamma_0^{-1}s),\ldots,\sigma(\gamma_i^{-1},s)^{-1}f(\gamma^{-1}_i s),\sigma(\gamma^{-1}_{i+1},s)^{-1},\ldots,\sigma(\gamma^{-1}_\bullet,s)^{-1})d\mu_\Omega(s) \ ,
$$
we can notice that $\rho^\bullet_{-1}(\sigma,f)=\textup{C}^\bullet_b(\sigma)$. A similar computation of \cite[Lemma 8.7.2]{monod:libro} readily implies that 
\begin{align*}
s^{\bullet+1}(\sigma,f)\delta^{\bullet}&=-\delta^{\bullet}s^{\bullet}(\sigma,f)+\sum_{i=0}^{\bullet}(\rho_{i-1}^{\bullet}(\sigma,f)-\rho^{\bullet}_i)=\\
&=-\delta^{\bullet}s^{\bullet}(\sigma,f)+\textup{C}^\bullet_b(\sigma)- \rho^\bullet_\bullet(\sigma,f) \ ,
\end{align*}
where $\delta^\bullet$ is the usual homogeneous coboundary operator. Since by Equation \eqref{eq:cochain:sigmaf} on the subcomplex of $H$-invariants cochains it holds
$$
\rho^\bullet_\bullet(\sigma,f)=\textup{C}^\bullet_b(\sigma^f) \ ,
$$
we get that 
$$
s^{\bullet+1}(\sigma,f)\delta^{\bullet}+\delta^{\bullet}s^{\bullet}(\sigma,f)=\textup{C}^\bullet_b(\sigma)- \textup{C}^\bullet_b(\sigma^f) \ ,
$$
and the claim follows. 
\end{proof}

Now we want to implement the pullback in terms of boundary maps. Consider a measure space $(Y,\nu)$ on which $H$ acts by preserving the measure class of $\nu$. Let $\phi:B(G) \times X \rightarrow Y$ be a boundary map for $\sigma$. Then we can consider the following $\textup{L}^\infty(X)$-valued pullback
$$
\textup{C}^\bullet(\phi):\mathcal{B}^\infty(Y^{\bullet+1};\bbR) \rightarrow \textup{L}^\infty_{\textup{w}^\ast}(B(G)^{\bullet+1};\textup{L}^\infty(X)) \ ,
$$
$$
\psi \mapsto \textup{C}^\bullet(\phi)(\psi)(\xi_0,\ldots,\xi_\bullet)(x):=\psi(\phi(\xi_0,x),\ldots,\phi(\xi_\bullet,x)) \ . 
$$

\begin{lem}
The map $\textup{C}^\bullet(\phi)$ is well-defined and induces a map between the subcomplexes of invariant vectors, that is
$$
\textup{C}^\bullet(\phi):\mathcal{B}^\infty(Y^{\bullet+1};\bbR)^H \rightarrow \textup{L}^\infty_{\textup{w}^\ast}(B(G)^{\bullet+1};\textup{L}^\infty(X))^\Gamma \ .
$$
Thus there exists a well-defined map in cohomology
$$
\textup{H}^\bullet(\phi):\textup{H}^\bullet(\mathcal{B}^\infty(Y^{\bullet+1};\bbR)^H) \rightarrow \textup{H}^\bullet_b(\Gamma;\textup{L}^\infty(X)) \ .
$$
\end{lem}

\begin{proof}
The fact that $\textup{C}^\bullet(\phi)$ is a cochain map which preserves boundedness is immediate. We are going to check that it restricts to a map between the subcomplexes of invariant vectors. First recall that we have an identification 
$$
 \textup{L}^\infty_{\textup{w}^\ast}(B(G)^{\bullet+1};\textup{L}^\infty(X))^\Gamma \cong \textup{L}^\infty(B(G)^{\bullet+1} \times X)^\Gamma \ ,
$$
where $\Gamma$ acts on the right-hand side diagonally. Let $\psi \in \mathcal{B}^\infty(Y^{\bullet+1};\bbR)^H$ and consider $\gamma \in \Gamma$. 
\begin{align*}
\gamma \cdot \textup{C}^\bullet(\phi)(\psi)(\xi_0,\ldots,\xi_\bullet)(x)&=\textup{C}^\bullet(\phi)(\psi)(\gamma^{-1}\xi_0,\ldots,\gamma^{-1}\xi_\bullet)(\gamma^{-1}x)=\\
&=\psi(\phi(\gamma^{-1}\xi_0,\gamma^{-1}x),\ldots,\phi(\gamma^{-1}\xi_\bullet,\gamma^{-1}x))=\\
&=\psi(\sigma(\gamma^{-1},x)\phi(\xi_0,x),\ldots,\sigma(\gamma^{-1},x)\phi(\xi_\bullet,x))=\\
&=\textup{C}^\bullet(\phi)(\psi)(\xi_0,\ldots,\xi_\bullet)(x) \ .
\end{align*}
In the computation above we used the $\sigma$-equivariance of $\phi$ to move from the second line to the third one and we concluded exploiting the $H$-invariance of $\psi$. 

Since by Theorem \ref{teor:resolution:boundary} the complex $(\textup{L}^\infty_{\textup{w}^\ast}(B(G)^{\bullet+1};\textup{L}^\infty(X))^\Gamma,\delta^\bullet)$ computes isometrically the continuous bounded cohomology $\textup{H}^\bullet_b(\Gamma;\textup{L}^\infty(X))$, the statement follows.  
\end{proof}

Now we can compose the map $\textup{H}^\bullet(\phi)$ with the map induced at the level of bounded cohomology by the integration
$$
\textup{I}_X^\bullet:\textup{L}^\infty_{\textup{w}^\ast}(B(G)^{\bullet+1};\textup{L}^\infty(X))^\Gamma \rightarrow \textup{L}^\infty(B(G)^{\bullet+1};\bbR)^\Gamma \ , 
$$
$$
\psi \mapsto \textup{I}_X^\bullet(\psi)(\xi_0,\ldots,\xi_\bullet):=\int_X \psi(\xi_0,\ldots,\xi_\bullet)d\mu_X(x) \ .
$$
It is easy to verify that $\textup{I}_X$ is a cochain map which preserves the boundedness and it restricts naturally to the subcomplexes of $\Gamma$-invariant vectors. In particular we have a well-defined map 
$$
\textup{I}_X^\bullet:\textup{H}^\bullet_b(\Gamma;\textup{L}^\infty(X)) \rightarrow \textup{H}^\bullet_b(\Gamma;\bbR) \ .
$$

\begin{deft}
Let $\sigma:\Gamma \times X \rightarrow H$ be a measurable cocycle with boundary map $\phi:B(G) \times X \rightarrow Y$. The \emph{cohomological pullback induced by $\phi$} is the map 
$$
\textup{H}^\bullet(\Phi^X):\textup{H}^\bullet(\mathcal{B}^\infty(Y^{\bullet+1};\bbR)^H)  \rightarrow \textup{H}^\bullet_b(\Gamma;\bbR) , \ \ \textup{H}^\bullet(\Phi^X)([\psi]):=\textup{I}_X^\bullet \circ \textup{H}^\bullet(\phi)([\psi]) \ .
$$
\end{deft}

We want to conclude this section by showing the compatibility between the two pullback maps. We have the following result which should be considered as an adaptation of \cite[Corollary 2.7]{burger:articolo}. 

\begin{lem}\label{lem:pullback:cocycle:boundary}
Let $\sigma:\Gamma \times X \rightarrow H$ be a measurable cocycle with a boundary map $\phi:B(G) \times X \rightarrow Y$. Given $\psi \in \calB^\infty(Y^{\bullet+1};\bbR)^H$, then 
$$
\textup{C}^\bullet(\Phi^X)(\psi) \in \textup{L}^\infty(B(G)^{\bullet+1};\mathbb{R})^\Gamma \ ,
$$  
is a natural representative of the class $\textup{H}^\bullet_b(\sigma)([\psi]) \in \textup{H}^\bullet_{b}(\Gamma;\bbR)$. 
\end{lem}

\begin{proof}
As a consequence of \cite[Proposition 1.2]{burger:articolo} we get the following commutative diagram
$$
\xymatrix{
\textup{H}^\bullet(\calB^\infty(Y^{\bullet+1};\bbR)^H) \ar[dd]^{\mathfrak{c}^\bullet} \ar[rrr]^{\textup{H}^\bullet(\Phi^X)} &&& \textup{H}^\bullet_b(\Gamma;\bbR) \\
\\
\textup{H}^\bullet_{cb}(H;\bbR) \ar[uurrr]^{\textup{H}^\bullet_b(\sigma)} &&& \ ,
}
$$
and the statement follows. 
\end{proof}

\section{The Borel invariant for a measurable cocycle}

After having introduced the machinery that we needed, we are finally ready to define the Borel invariant for a measurable cocycle. We are going to do it only for a torsion-free non-uniform lattice. The case of uniform lattices is actually easier and can be obtained from the non-uniform one with some slight modifications. 

Let $\Gamma \leq \psl(2,\bbC)$ be a torsion-free non-uniform lattice and denote by $M=\Gamma \backslash \bbH^3_{\bbR}$ the quotient manifold associated to $\Gamma$. We consider $N \subset M$ a \emph{compact core} of $M$, that is the complement of the disjoint union of some horocyclic neighborhoods of the cusps. Since the fundamental group of the boundary $\partial N$ is virtually abelian, hence amenable, by \cite{BBFIPP} there exists an isometric isomorphism $$\textup{H}^\bullet_b(j):\textup{H}^\bullet_b(M, M \setminus N;\bbR) \rightarrow \textup{H}^\bullet_b(M;\bbR) \ , $$ provided that $\bullet \geq 3$. Here $\textup{H}^\bullet_b(j)$ is the map induced by the inclusion of the pair $j:(M,\varnothing) \rightarrow (M,M\setminus N)$. 

Thus we can consider the composition 
\begin{equation}\label{eq:composition}
\textup{H}^\bullet_b(\Gamma ;\bbR) \cong \textup{H}^\bullet_b(M;\bbR) \rightarrow \textup{H}^\bullet_b(M,M \setminus N;\bbR) \cong \textup{H}^\bullet_b(N,\partial N;\bbR) \ ,
\end{equation}
where the isomorphism on the left is due to the Gromov's Mapping Theorem \cite{Grom82,Ivanov}, the map in the middle is the inverse of $\textup{H}^\bullet_b(j)$ and the isomorphism on the right holds since the pairs are homotopically equivalent. We are going to denote the composition of Equation \eqref{eq:composition} by $\textup{J}^\bullet$. 

\begin{deft}\label{def:borel:invariant}
Let $\Gamma \leq \psl(2,\bbC)$ be a torsion-free non-uniform lattice and let $(X,\mu_X)$ be a standard Borel probability space. Given a measurable cocycle $\sigma:\Gamma \times X \rightarrow \psl(n,\bbC)$ we define the \emph{Borel invariant associated to $\sigma$} as 
$$
\beta_n(\sigma):=\langle \textup{comp}_{N,\partial N}^3 \circ \textup{J}^3 \circ \textup{H}^3_b(\sigma)(\beta_b(n)), [N,\partial N] \rangle \ .
$$
Here $\langle \ \cdot , \cdot \ \rangle$ is the Kronecker pairing, $\textup{comp}^3_{N,\partial N}$ is the comparison map of the pair $(N,\partial N)$ and $[N,\partial N] \in \textup{H}_3(N,\partial N;\bbR)$ is a fixed relative fundamental class. 
\end{deft}

The first natural question could be how the Borel invariant depends on the choice of the compact core $N$. Since the Borel invariant will satisfy the integral formula of Proposition \ref{prop:formula}, this will imply in particular that $\beta_n(\sigma)$ actually does not depend on the choice of $N$. 

The first thing that we are going to show is that our definition of the Borel invariant actually extends the one given by Bucher, Burger and Iozzi for representations.

\begin{prop}\label{prop:borel:representation}
Let $\Gamma \leq \psl(2,\bbC)$ be a torsion-free non-uniform lattice and let $\rho:\Gamma \rightarrow \psl(n,\bbC)$ be a representation. Given a standard Borel probability $\Gamma$-space $(X,\mu_X)$, if $\sigma_\rho:\Gamma \times X \rightarrow \psl(n,\bbC)$ is the measurable cocycle associated to $\rho$, it holds
$$
\beta_n(\sigma_\rho)=\beta_n(\rho) \ .
$$
\end{prop}

\begin{proof}
The proof is an immediate consequence of Lemma \ref{lem:pullback:representation}. Indeed we have
\begin{align*}
\beta_n(\sigma_\rho)&=\langle \textup{comp}^3_{N, \partial N} \circ \textup{J}^3 \circ \textup{H}^3_b(\sigma_\rho)(\beta_b(n)),[N,\partial N] \rangle\\
&=\langle \textup{comp}^3_{N, \partial N} \circ \textup{J}^3 \circ \textup{H}^3_b(\rho)(\beta_b(n)),[N,\partial N] \rangle=\beta(\rho) \ , 
\end{align*}
and the statement follows. 
\end{proof}

We now want to show that the Borel invariant remains unchanged along the $\psl(n,\bbC)$-cohomology class of a fixed measurable cocycle. 

\begin{prop}\label{prop:borel:cohomology}
Let $\Gamma \leq \psl(2,\bbC)$ be a torsion-free non-uniform lattice and let $(X,\mu_X)$ be a standard Borel probability $\Gamma$-space. Given a measurable cocycle $\sigma:\Gamma \times X \rightarrow \psl(n,\bbC)$ and any measurable map $f:X \rightarrow H$ it holds
$$
\beta_n(\sigma^f)=\beta_n(\sigma) \ .
$$
\end{prop}

\begin{proof}
The statement is a direct consequence of Lemma \ref{lem:pullback:cohomology}. Indeed we have
\begin{align*}
\beta_n(\sigma^f)&=\langle \textup{comp}^3_{N, \partial N} \circ \textup{J}^3 \circ \textup{H}^3_b(\sigma^f)(\beta_b(n)),[N,\partial N] \rangle\\
&=\langle \textup{comp}^3_{N, \partial N} \circ \textup{J}^3 \circ \textup{H}^3_b(\sigma)(\beta_b(n)),[N,\partial N] \rangle=\beta(\sigma) \ .
\end{align*}
This concludes the proof. 
\end{proof}

Using jointly both Proposition \ref{prop:borel:representation} and Proposition \ref{prop:borel:cohomology}, we can immediately argue the following

\begin{cor}\label{cor:one:direction}
Let $\Gamma \leq \psl(2,\bbC)$ be a torsion-free non-uniform lattice. Then it holds
$$
|\beta_n(\rho)|={n+1 \choose 3}\vol(M) \ ,
$$
for every cocycle $\sigma:\Gamma \times X \rightarrow \psl(n,\bbC)$ which is cohomologous either to the restriction to $\Gamma$ of the irreducible representation $\pi_n$ or to its complex conjugated. 
\end{cor}

\begin{proof}
Bucher, Burger and Iozzi \cite{BBIborel} proved that the restriction to $\Gamma$ of the irreducible representation has maximal Borel invariant. By Proposition \ref{prop:borel:representation} it follows
$$
\beta_n(\sigma_{\pi_n})={n+1 \choose 3}\vol(M) \ . 
$$
By changing suitably the sign, the same holds also for the restriction of the complex conjugated $\overline{\pi}_n$. Since by Proposition \ref{prop:borel:cohomology} the Borel invariant is constant along cohomology classes, the desired statement follows.  
\end{proof}

We now prove the natural bound on the absolute value on the Borel invariant. 

\begin{prop}\label{prop:borel:bound}
Let $\Gamma \leq \psl(2,\bbC)$ be a torsion-free non-uniform lattice and let $(X,\mu_X)$ be a standard Borel probability $\Gamma$-space. Given a measurable cocycle $\sigma:\Gamma \times X \rightarrow \psl(n,\bbC)$ it holds
$$
|\beta_n(\sigma)| \leq {n+1 \choose 3} \vol(M) \ .
$$
\end{prop}

\begin{proof}
Inspired by \cite{BBIborel} we have the following commutative diagram
$$
\xymatrix{
\textup{H}^3_{cb}(\psl(n,\bbC);\bbR) \ar[d]^{\textup{H}^3_b(\sigma)} \\
\textup{H}^3_b(\Gamma;\bbR) \ar[d]^{\textup{J}^3} \ar[rr]^{\textup{trans}^3_{\Gamma}} && \textup{H}^3_{cb}(\psl(2,\bbC);\bbR) \ar[dd]^{\textup{comp}^3_{\psl(2,\bbC)}}\\
\textup{H}^3_b(N,\partial N;\bbR) \ar[d]^{\textup{comp}^3_{N,\partial N}} &&  \\
\textup{H}^3(N,\partial N;\bbR) \ar[rr]^{\tau^3_{\textup{dR}}} && \textup{H}^3_c(\psl(2,\bbC);\bbR) \ .
}
$$

Recall that $\textup{H}^3_{cb}(\psl(2,\bbC);\bbR)$ is a one dimensional vector space generated by the bounded Borel class $\beta_b(2)$ (which coincides with the volume class). Thus there must exists a real number $\lambda \in \bbR$ such that 
$$
\textup{comp}^3_{\psl(2,\bbC)} \circ \textup{trans}^3_\Gamma \circ \textup{H}^3_b(\sigma)(\beta_b(n))=\lambda \beta(2) \ .
$$
If we consider now a differential form $\omega_{N,\partial N} \in \textup{H}^3(N, \partial N;\bbR)$ such that its evaluation with the fundamental class $[N,\partial N]$ gives us back the volume $\vol(M)$, it must hold $\tau^3_\textup{dR}(\omega_{N,\partial N})=\beta_b(2)$. As a consequence we get that
$$
\tau^3_{\textup{dR}} \circ \textup{comp}^3_{N,\partial N} \circ \textup{J}^3 \circ \textup {H}^3_b(\sigma)(\beta_b(n))=\lambda \tau^3_{\textup{dR}}(\omega_{N,\partial N}) \ .
$$
Since the transfer map $\tau^3_{\textup{dR}}$ is injective (see \cite{bucher2:articolo} for instance), we must have
$$
\textup{comp}^3_{N,\partial N} \circ \textup{J}^3 \circ \textup {H}^3_b(\sigma)(\beta_b(n))=\lambda \omega_{N,\partial N} \ .
$$
Evaluating both sides on the fundamental class $[N,\partial N]$ we get back
$$
\beta_n(\rho)=\langle \textup{comp}^3_{N,\partial N} \circ \textup{J}^3 \circ \textup {H}^3_b(\sigma)(\beta_b(n)),[N,\partial N] \rangle=\lambda \vol(M) \ .
$$
At the same time it holds
$$
| \lambda | \leq \frac{ \lVert \textup{trans}^3_\Gamma \circ \textup{H}^3_b(\sigma)(\beta_b(n)) \rVert_\infty}{\beta_b(2)}={n+1 \choose 3} \ ,
$$
since both $\textup{trans}^3_\Gamma$ and $\textup{H}^3_b(\sigma)$ are norm non-increasing. From the previous estimate it follows
$$
\left| \frac{\beta_n(\sigma)}{\vol(M)} \right| \leq {n+1 \choose 3} \ ,
$$
and this concludes the proof. 
\end{proof}

In virtue of the previous proposition we can give the following 

\begin{deft}\label{def:maximal:cocycle}
Let $\Gamma \leq \psl(2,\bbC)$ be a torsion-free non-uniform lattice. Let $(X,\mu_X)$ be a standard Borel probability $\Gamma$-space. A measurable cocycle $\sigma:\Gamma \times X \rightarrow \psl(n,\bbC)$ is \emph{maximal} if it holds $|\beta_n(\sigma)|={n+1 \choose 3}\vol(M)$. 
\end{deft}

We conclude the section by proving that parabolic cocycles cannot be maximal.

\begin{proof}[Proof of Proposition \ref{prop:parabolic:cocycle}]
Suppose that $\sigma:\Gamma \times X \rightarrow \psl(n,\bbC)$ is cohomologous to a cocycle $\sigma_0$ whose image is contained in a parabolic subgroup. Since the Borel invariant does not change along the $\psl(n,\bbC)$-cohomology class (Proposition \ref{prop:borel:cohomology}), we can directly consider the cocycle $\sigma_0$. By hypothesis there exists a proper parabolic subgroup $P_0 \leq \psl(n,\bbC)$ such that $\sigma_0:\Gamma \times X \rightarrow P_0$. As a consequence we have the following commutative diagram 
$$
\xymatrix{
\textup{H}^3_{cb}(\psl(n,\bbC);\bbR) \ar[rr]^{\textup{H}^3_b(\sigma_0)} \ar[d]^{\textup{H}^3_{cb}(i_0)} && \textup{H}^3_{b}(\Gamma;\bbR) \\
\textup{H}^3_{cb}(P_0;\bbR) \ar[urr] \ , \\ 
}
$$
where $\textup{H}^3_{cb}(i_0)$ is the restriction map induced by the inclusion $i_0: P_0 \rightarrow \psl(n,\bbC)$. If we denote by $\beta_{P_0}(n):=\textup{H}^3_{cb}(i_0)(\beta_b(n))$, the commutativity of the diagram above implies that 
$$
\textup{H}^3_b(\sigma_0)(\beta_b(n))=\textup{H}^3_b(\sigma_0)(\beta_{P_0}(n)) \ .
$$
Recall that $P_0$ is the stabilizer of some incomplete flag $0 \subset F_{i_1} \subset \ldots \subset F_{i_r}=\bbC^n$. We define $n_1=i_1, \ n_2=i_2 - i_1, \ \ldots , i_r=i_r-i_{r-1}$. Then $(n_1,\ldots,n_r)$ is the partition associated to $P_0$. It is easy to prove that we can write the class $\beta_{P_0}(n)$ as the following sum 
$$
\beta_{P_0}(n)=\sum_{i=1}^r \beta_b(n_i) \ ,
$$
where the sum on the right-hand side lies in the group $\textup{H}^3_{cb}(\prod_{i=1}^r \textup{GL}(n_i;\bbC);\bbR)$, which is isomorphic to $\textup{H}^3_{cb}(P_0;\bbR)$. As a consequence 
$$
|\beta_n(\sigma_0)| \leq \sum_{i=1}^r {n_i+1 \choose 3}\vol(M) \ ,
$$
and the statement follows. 
\end{proof}

\section{Rigidity of the Borel invariant}

In this section we are going to prove the main theorem of the paper. We are going to show that the Borel function is rigid, i.e. it attains its maximum value only on the cohomology class of the cocycle induced by the irreducible representation. The proof will rely crucially on the existence of a boundary map for maximal cocycles.

We start showing the integral formula stated in the introduction.

\begin{proof}[Proof of Proposition \ref{prop:formula}]
Following the proof of Proposition \ref{prop:borel:bound}, from the fact that $\textup{H}^3_{cb}(\psl(2,\bbC);\bbR)$ is one dimensional and generated by $\beta_b(2)$, we argue
\begin{equation}\label{eq:cochain:formula}
\textup{trans}^3_\Gamma \circ \textup{H}^3_b(\sigma)(\beta_b(n))=\frac{\beta_n(\sigma)}{\vol(M)} \beta_b(2) \ .
\end{equation}

Since we assumed the existence of a boundary map then we can substitue $\textup{H}^3_{b}(\sigma)(\beta_b(n))$ with the class $\textup{H}^3(\Phi^X)([B_n])$, in virtue of Lemma \ref{lem:pullback:cocycle:boundary}. 

Recall that $\Gamma$ acts doubly ergodically on the boundary $\bbP^1(\bbC)$. This implies that there are no alternating cochains in degree $2$, that means $\textup{L}^\infty_{\textup{alt}}(\bbP^1(\bbC)^2;\bbR)^\Gamma=0$. Thus we can rewrite Equation \eqref{eq:cochain:formula} in terms of cochains as follows
$$
\widehat{\textup{trans}}^3_\Gamma \circ \textup{C}^3(\Phi^X)(B_n)=\frac{\beta_n(\sigma)}{\vol(M)}B_2 \ .
$$
Since $B_2$ coincides with the hyperbolic volume function $\vol$, we get the desired formula. The fact that the formula holds for every $4$-tuple of points $(\xi_0,\ldots,\xi_3)$ is a consequence of \cite[Proposition 4.2]{bucher2:articolo}. 
\end{proof}

\begin{oss}
The previous formula shows for instance that the Borel invariant does not depend on the choice of the compact compact core $N$ used in its definition. 
\end{oss}

Before proving the main theorem of the section, we need to show that a maximal measurable cocycle admits a boundary map.

\begin{prop}\label{prop:existence:boundary}
Let $\Gamma \leq \psl(2,\bbC)$ be a torsion-free non-uniform lattice and let $(X,\mu_X)$ be a standard Borel probability $\Gamma$-space. A maximal measurable cocycle $\sigma:\Gamma \times X \rightarrow \psl(n,\bbC)$ admits a boundary map.
\end{prop}

\begin{proof}
We denote by $\mathcal{M}^1(\scrF(n,\bbC))$ the space of probability measures on $\scrF(n,\bbC)$. Since the space $\bbP^1(\bbC) \times X$ is $\Gamma$-amenable by \cite[Proposition 4.3.4]{zimmer:libro}, there exists a measurable $\sigma$-equivariant map 
$$
\phi:\bbP^1(\bbC) \times X \rightarrow \mathcal{M}^1(\scrF(n,\bbC)) \ ,
$$
by \cite[Proposition 4.3.9]{zimmer:libro}. Under the maximality assumption we are going to show that the essential image of $\phi$ lies in the set of Dirac masses. 
Given any $\xi_0,\ldots,\xi_3 \in \bbP^1(\bbC)$ and $x \in X$, we are going to denote by $\phi(\xi_0,x) \otimes \ldots \otimes  \phi(\xi_3,x)$ the product measure on $\scrF(n,\bbC)$. Similarly we are going to write
$$
\langle \phi(\xi_0,x) \otimes \ldots \otimes \phi(\xi_3,x) , B_n \rangle:=\int_{\scrF(n,\bbC)^4} B_n(F_0,\ldots,F_3) d\phi(\xi_0,x)(F_0) \ldots d\phi(\xi_3,x)(F_3) \ .
$$
Following the same strategy of \cite[Section 8.2]{BBIborel} we get that 
\begin{small}
\begin{equation}\label{eq:boundary:measure}
\frac{\beta_n(\sigma)}{\vol(M)}\vol(\xi_0,\ldots,\xi_3)=\int_{\Gamma \backslash \psl(2,\bbC)}\int_X \langle \phi(\overline{g}\xi_0,x) \otimes \ldots \otimes \phi(\overline{g}\xi_3,x) , B_n \rangle d\mu_X(x)d\mu(\overline{g}) \ .
\end{equation}
\end{small}
for every $\xi_0,\ldots,\xi_3 \in \bbP^1(\bbC)$. Without loss of generality we can assume that $\sigma$ is positively maximal, that is $\beta_n(\sigma)={n+1 \choose 3}\vol(M)$. The case when $\sigma$ is negatively maximal is similar and we omit it. Since Equation \eqref{eq:boundary:measure} holds for every $4$-tuple $(\xi_0,\ldots,\xi_3)$, we can consider the vertices of a regular ideal tetrahedron which is positively oriented. By the maximality assumption we get that
$$
{n+1 \choose 3}\nu_3=\int_{\Gamma \backslash \psl(2,\bbC)}\int_X \langle \phi(\overline{g}\xi_0,x) \otimes \ldots \otimes \phi(\overline{g}\xi_3,x) , B_n \rangle d\mu_X(x)d\mu(\overline{g}) \
$$
The bound on the Borel function $B_n$ implies that
$$
\langle \phi(\overline{g}\xi_0,x) \otimes \ldots \otimes \phi(\overline{g}\xi_3,x) , B_n \rangle={ n+1 \choose 3} \nu_3 \ ,
$$
for almost every $\overline{g} \in \Gamma \backslash \psl(2,\bbC), x \in X$. Thanks to the equivariance of $\phi$ we get that 
$$
\langle \phi(g\xi_0,x) \otimes \ldots \otimes \phi(g\xi_3,x) , B_n \rangle={ n+1 \choose 3} \nu_3 \ ,
$$
for almost every $g \in \psl(2,\bbC), x \in X$. The previous equation can be rewritten as 
$$
\int_{\scrF(n,\bbC)^4} B_n(F_0,\ldots,F_3) d\phi(g\xi_0,x)(F_0) \ldots d\phi(g\xi_3,x)(F_3)={n+1 \choose 3}\nu_3 \ ,
$$
for almost every $g \in \psl(2,\bbC), x \in X$. Fix now a triple of flags $F_0,F_1,F_2 \in \scrF(n,\bbC)$ such that 
$$
B(F_0,F_1,F_2,F_3)={n+1 \choose 3}\nu_3 \ ,
$$
holds for $\phi(g\xi_3,x)$-almost every $F_3 \in \scrF(n,\bbC)$. By \cite[Corollary 20]{BBIborel} the flag $F_3$ is unique and $\phi(g\xi_3,x)$ must be a Dirac mass. Since the same reasoning applies for almost every $g \in \psl(2,\bbC)$ and almost every $x \in X$, the statement follows. 
\end{proof}

We are finally ready to show the main theorem 

\begin{proof}[Proof of Theorem \ref{teor:rigidity}]
One direction is proved by Corollary \ref{cor:one:direction}. We are going to prove the other implication. Up to composing the cocycle with the complex conjugation, we are going to suppose that $\sigma$ is positively maximal. By Proposition \ref{prop:existence:boundary} $\sigma$ admits a boundary map $\phi:\bbP^1(\bbC) \times X \rightarrow \scrF(n,\bbC)$. 

Consider $\xi_0,\ldots,\xi_3 \in \bbP^1(\bbC)$ vertices of a regular ideal tetrahedron which is positively oriented. For almost every $x \in X$ define the slice
$$
\phi_x: \bbP^1(\bbC) \rightarrow \scrF(n,\bbC) \ , \ \ \ \phi_x(\xi):=\phi(\xi,x).
$$
Notice that these maps are all measurable by \cite[Lemma 2.6]{fisher:morris:whyte}. 

As a consequence of Proposition~\ref{prop:formula} and of the maximality of $\beta_n(\sigma)$ we have that
$$
\int_{\Gamma \backslash \psl(2,\bbC)} \int_X B_n(\phi_x(\overline{g} \xi_0),\ldots,\phi_x(\overline{g} \xi_3))d\mu_X(x)d\mu (\overline{g})={n+1 \choose 3}\nu_3 \ .
$$
The equation above implies that for almost every $x \in X$ and almost every $\overline{g} \in \Gamma \backslash \psl(2,\bbC)$ it holds
$$
B_n(\phi_x(\bar g \xi_0),\ldots,\phi_x(\bar g \xi_3))={n+1 \choose 3}\nu_3
$$
and by the $\sigma$-equivariance of $\phi$ the equality can be extended to almost every $x \in X$ and almost every $g \in \psl(2,\bbC)$. The same equality will hold also if $(\xi_0,\ldots,\xi_3)$ are vertices of a regular ideal tetrahedron which is negatively oriented. Hence for almost every $x \in X$ the measurable map $\phi_x$ satisfies the hypothesis of Proposition \ref{prop:maximal:map} and hence there must exists $f(x) \in \psl(n,\bbC)$ such that
\[
\phi_x(\xi)=f(x)\mathcal{V}_n(\xi),
\]
where $\mathcal{V}_n:\bbP^1(\bbC) \rightarrow \scrF(n,\bbC)$ is the Veronese embedding. In this way we get a map $f:X \rightarrow \psl(n,\bbC)$ which is measurable by the measurability of $\phi$. 

By the equivariance of the Verenose embedding with respect to the irreducible representation $\pi_n$, for every $\gamma \in \Gamma$ we have that 
\[
\phi_{\gamma x}(\gamma \xi)=f(\gamma x)\mathcal{V}_n(\gamma \xi)=f(\gamma x)\pi_n(\gamma) \mathcal{V}_n(\xi)
\]
and by the equivariance of $\sigma$ it holds at the same time 
\[
\phi_{\gamma x}(\gamma \xi)=\sigma(\gamma,x)\phi_x(\xi)=\sigma(\gamma,x)f(x)\mathcal{V}_n(\xi).
\]
The equations above imply
\[
\pi_n(\gamma)=f^{-1}(\gamma x)\sigma(\gamma,x)f(x),
\]
and the theorem is proved.
\end{proof}

We want to conclude the paper with some remark about the machinery developed so far. It should be clear to the reader that all the techniques studied in Section \ref{sec:pullback:cocycle} have a wider range of application. Indeed, as already written in the introduction, this paper has been a source of inspiration to the author who has applied this machinery to the study of measurable cocycles associated to lattices in semisimple Lie groups obtaining several rigidity results similar to Theorem \ref{teor:rigidity} (see \cite{moraschini:savini,moraschini:savini:2,savini:surface,sarti:savini,sarti:savini:2}). The crucial aspect of this paper is that the Borel invariant is the first example of \emph{multiplicative constant} (see \cite[Definition 3.20]{moraschini:savini}) studied by the author. Indeed it satisfies the formula of Proposition \ref{prop:formula}, that the author has later named \emph{multiplicative formula}. 

\bibliographystyle{amsalpha}

\bibliography{biblionote}

\newcommand{\etalchar}[1]{$^{#1}$}
\providecommand{\bysame}{\leavevmode\hbox to3em{\hrulefill}\thinspace}
\providecommand{\MR}{\relax\ifhmode\unskip\space\fi MR }
\providecommand{\MRhref}[2]{%
  \href{http://www.ams.org/mathscinet-getitem?mr=#1}{#2}
}
\providecommand{\href}[2]{#2}
\begin{thebibliography}{BBF{\etalchar{+}}14}

\bibitem[BBF{\etalchar{+}}14]{BBFIPP}
M.~Bucher, M.~Burger, R.~Frigerio, A.~Iozzi, C.~Pagliantini, and M.~B.
  Pozzetti, \emph{Isometric properties of relative bounded cohomology}, J.
  Topol. Anal. \textbf{6} (2014), no.~1, 1--25.

\bibitem[BBI13]{bucher2:articolo}
M.~Bucher, M.~Burger, and A.~Iozzi, \emph{A dual interpretation of the
  {G}romov--{T}hurston proof of {M}ostow rigidity and volume rigidity for
  representations of hyperbolic lattices}, Trends in harmonic analysis,
  Springer INdAM Ser., Springer, Milan, 2013, pp.~47--76.

\bibitem[BBI18]{BBIborel}
\bysame, \emph{The bounded borel class and complex representations of
  $3$-manifold groups}, Duke Math. \textbf{167} (2018), no.~17, 3129--3169.

\bibitem[BFS13a]{sauer:companion}
U.~Bader, A.~Furman, and R.~Sauer, \emph{Efficient subdivision in hyperbolic
  groups and applications}, Groups Geom. Dyn. \textbf{7} (2013), 263--292.

\bibitem[BFS13b]{sauer:articolo}
\bysame, \emph{Integrable measure equivalence and rigidity of hyperbolic
  lattices}, Invent. Math. \textbf{194} (2013), 313--379.

\bibitem[BI02]{burger:articolo}
M.~Burger and A.~Iozzi, \emph{Boundary maps in bounded cohomology},
  Geom.~Funct.~Anal. \textbf{12} (2002), no.~2, 281--292, Appendix to
  ``Continuous bounded cohomology and applications to rigidity theory" by M.
  Burger and N. Monod.

\bibitem[Blo00]{bloch:libro}
S.~J. Bloch, \emph{Higher regulators, algebraic k-theory, and zeta functions of
  elliptic curves}, CRM Monograph Series, vol.~11, American Mathematical
  Society, Providence, 2000.

\bibitem[BM02]{burger2:articolo}
M.~Burger and N.~Monod, \emph{Continuous bounded cohomology and applications to
  rigidity theory}, Geom. Funct. Anal. \textbf{12} (2002), 219--280.

\bibitem[FM77]{feldman:moore}
J.~Feldman and C.~C. Moore, \emph{Ergodic equivalence relations, cohomology,
  and von neumann algebras}, Trans.~Amer.~Math.~Soc. \textbf{234} (1977),
  289--324.

\bibitem[FMW04]{fisher:morris:whyte}
D.~Fisher, D.~W. Morris, and K.~Whyte, \emph{Nonergodic actions, cocycles and
  superrigidity}, New York J.~Math. \textbf{10} (2004), 249--269.

\bibitem[Fur81]{furst:articolo}
H.~Furstenberg, \emph{Rigidity and cocycles for ergodic actions of semisimple
  lie groups}, Bourbaki Seminar \textbf{1979/80} (1981), 273--292, Lecture
  Notes in Math.

\bibitem[Gro82]{Grom82}
M.~Gromov, \emph{Volume and bounded cohomology}, Publ. Math. Inst. Hautes
  \'Etudes Sci. \textbf{56} (1982), 5--99.

\bibitem[Gui80]{guichardet}
A.~Guichardet, \emph{Cohomologie des groupes topologiques et de alg{\`e}bres de
  lie}, Textes Mat{\'e}matiques, vol.~2, Springer, CEDIC Paris, 1980.

\bibitem[Iva87]{Ivanov}
N.~V. Ivanov, \emph{Foundation of the theory of bounded cohomology},
  J.~Soviet.~Math. \textbf{37} (1987), 1090--1114.

\bibitem[Mar75]{margulis:super}
G.~A. Margulis, \emph{Discrete groups of motions of manifolds of nonpositive
  curvature}, Prooceedings of the International Congress of Mathematicians
  (Vancouver, B:~C., 1974) \textbf{2} (1975), 21--34, Canad. Math. Congress,
  Montreal Que.

\bibitem[Mon01]{monod:libro}
N.~Monod, \emph{Continuous bounded cohomology of locally compact groups},
  Lecture notes in Mathematics, no. 1758, Springer-Verlag, Berlin, 2001.

\bibitem[Mos68]{mostow68:articolo}
G.~D. Mostow, \emph{Quasi-conformal mappings in $n$-space and the rigidity of
  the hyperbolic space forms}, Inst. Hautes \'Etudes Sci. Publ. Math.
  \textbf{34} (1968), 53--104.

\bibitem[Mos73]{Most73}
\bysame, \emph{Strong rigidity of locally symmetric spaces}, Ann. of Math.
  Studies, vol.~78, Princeton University Press, 1973.

\bibitem[MSa]{moraschini:savini}
M.~Moraschini and A.~Savini, \emph{A matsumoto/mostow result for zimmer's
  cocycles of hyperbolic lattices}, To appear in Transform.~Groups,
  https://doi.org/10.1007/s00031-020-09630-z.

\bibitem[MSb]{moraschini:savini:2}
\bysame, \emph{Multiplicative constants and maximal measurable cocycles in
  bounded cohomology}, to appear on Ergodic Theory Dynam. Systems, 2021+, DOI
  number:10.1017/etds.2021.91.

\bibitem[MS04]{MonShal0}
N.~Monod and Y.~Shalom, \emph{Cocycle superrigidity and bounded cohomology for
  negatively curved spaces}, J. Differential Geom. \textbf{67} (2004),
  395--455.

\bibitem[MS06]{MonShal}
\bysame, \emph{Orbit equivalence rigidity and bounded cohomology}, Ann. of
  Math. (2) \textbf{164} (2006), 825--878.

\bibitem[Sav]{savini:tetr}
A.~Savini, \emph{Rigidity at infinity for the borel function of the tetrahedral
  reflection lattice}, to appear on Algeb. Geom. Topology, 2021+,
  https://arxiv.org/pdf/1906.02620.

\bibitem[Sav21]{savini:surface}
\bysame, \emph{Algebraic hull of maximal measurable cocycles of surface groups
  inside hermitian lie groups}, Geom. Dedicata \textbf{213} (2021), no.~1,
  375--400.

\bibitem[SSa]{sarti:savini:2}
F.~Sarti and A.~Savini, \emph{Finite reducibility of maximal infinite
  dimensional measurable cocycles of complex hyperbolic lattices},
  https://arxiv.org/pdf/2005.10529.pdf.

\bibitem[SSb]{sarti:savini}
\bysame, \emph{Superrigidity of maximal measurable cocycles of complex
  hyperbolic lattices}, to appear on Math. Z., 2021+, DOI number:
  10.10072Fs00209-021-02801-y.

\bibitem[Zim80]{zimmer:annals}
R.~J. Zimmer, \emph{Strong rigidity for ergodic actions of semisimple lie
  groups}, Ann.~of Math. \textbf{112} (1980), no.~3, 511--529.

\bibitem[Zim84]{zimmer:libro}
\bysame, \emph{Ergodic theory and semisimple groups}, Monographs in
  Mathematics, vol.~81, Birkh{\"a}user Verlag, Basel, 1984.

\end{thebibliography}

\end{document}